\documentclass[a4paper,11pt]{article}
\usepackage{soul}
\usepackage{amsmath, amssymb, amsfonts, amsthm}
\usepackage{mathrsfs}
\usepackage{enumitem}
\usepackage{stmaryrd}
\usepackage{color}
\usepackage{graphicx}
\usepackage{hyperref}
\usepackage{wrapfig}
\usepackage{cleveref}

\theoremstyle{plain}
\newtheorem{theorem}{Theorem}[section]
\newtheorem{definition}[theorem]{Definition}

\newtheorem{lemma}[theorem]{Lemma}

\newtheorem{corollary}[theorem]{Corollary}
\theoremstyle{remark}
\newtheorem{remark}[theorem]{Remark}

\usepackage{mathrsfs, stmaryrd, mathtools}

\DeclarePairedDelimiterX\intff[2]{[}{]}{#1,#2}
\DeclarePairedDelimiterX\intfo[2]{[}{)}{#1,#2}
\DeclarePairedDelimiterX\intof[2]{(}{]}{#1,#2}
\DeclarePairedDelimiterX\intoo[2]{(}{)}{#1,#2}
\DeclarePairedDelimiter{\pars}{(}{)}
\DeclarePairedDelimiter{\bracks}{[}{]}

\DeclarePairedDelimiter{\braces}{\lbrace}{\rbrace}
\DeclarePairedDelimiterX{\setof}[2]{\lbrace}{\rbrace}{#1\,{:}\,#2}
\DeclarePairedDelimiterX{\bracksof}[2]{[}{]}{#1\,\delimsize\vert\,#2}
\DeclarePairedDelimiterX{\parsof}[2]{(}{)}{#1\,\delimsize\vert\,#2}
\DeclarePairedDelimiterXPP\lnorm[2]{}\lVert\rVert{_{#1}}{#2}

\usepackage{bbm}
\def\Cap{\mathtt{Cap}}

\def\1{\mathbbm{1}}

\def\dif{\mathrm{d}}
\def\diam{\mathrm{diam}}

\title{Minkowski sum of fractal percolation and random sets}
\author{Tianyi Bai, Xinxin Chen, Yuval Peres}

\begin{document}
\maketitle

\begin{abstract}
In this paper, we prove that hitting probability of Minkowski sum of fractal percolations can be characterized by capacity. Then we extend this result to Minkowski sum of general random sets in $\mathbb Z^d$, including ranges of random walks and critical branching random walks, whose hitting probabilities are described by Newtonian capacity individually.
\end{abstract}

\section{Introduction}
Capacity is a concept from potential theory that is useful in studying long-range hitting probability of a set. Consider a finite set $A\subset\mathbb Z^d$, for each $0<\beta<d$, the (discrete) Newtonian capacity with parameter $\beta$ is
\begin{equation}\label{eq:def_cap}
\Cap_\beta (A):=\pars*{\inf_{\mu}\sum_{x,y\in A}\mu(x)\mu(y)(|x-y|\vee 1)^{-\beta}}^{-1},
\end{equation}
where the infimum is taken over all probability measures $\mu$ on $A$, and $|x-y|$ means Euclidean distance. 
Let $(S_n)$ be a simple random walk in $\mathbb Z^d$ started from the origin, it is well-known that (see Lawler and Limic \cite[Proposition 6.5.1]{lawler2010random}) if $0\in A$ and $|x|\ge2\diam(A)$,
\begin{equation}\label{eq:eg_cap}
\mathbb P_0((S_n)\text{ hits }(x+A))\asymp |x|^{2-d}\Cap_{d-2}(A),
\end{equation}
where $\diam(A)$ is the diameter of $A$, and the notation $f\asymp g$ means $C_1 g\le f\le C_2 g$ for some $C_1,C_2>0$ only depending on dimension $d$.

Recently, it is shown in Asselah, Okada, Schapira and Sousi \cite[Theorem 1.6]{asselah2023branching} that hitting probability estimates for Minkowski sum of simple random walks can be given in a similar pattern: Let $d>2k$, let $\mathcal R_1,\dots,\mathcal R_k$ be the ranges of independent simple random walks,  i.e.
\[
\mathcal R_i=\{S^{(i)}_0,S^{(i)}_1,\dots\},
\]
where for each $i=1,2,\dots,k$, $(S^{(i)}_n)$ is a simple random walk started from $0\in\mathbb Z^d$, independent of each other. For every finite set $A\subset\mathbb Z^d$ containing the origin, when $|x|\ge 2\diam(A)$, 
\[
\mathbb P\pars*{(\mathcal R_1+\dots+\mathcal R_k)\cap(x+A)\ne\emptyset}\asymp |x|^{2k-d}\Cap_{d-2k}(A),
\]
where Minkowski sum is defined by
\[
A+B:=\setof{a+b}{a\in A,b\in B}.
\]

In this paper, we prove this phenomenon for any random set with property similar to \eqref{eq:eg_cap}.

\begin{theorem}\label{thm:rw+rw}
Let $d,k\ge 1$, let $\alpha_1,\beta_1,\dots,\alpha_k,\beta_k\in(0,d)$ such that 
\[
\alpha_1+\dots+\alpha_k>(k-1)d,
\]
\[
\beta_1+\dots+\beta_k>(k-1)d.
\]
If $\mathcal R_1,\dots,\mathcal R_k$ are independent random sets in $\mathbb Z^d$ with the following property: for every finite set $A\subset\mathbb Z^d$ containing the origin and every $|x|\ge 2\diam(A)$,
we have
\[
{\mathbb P(\mathcal R_i\cap (x+A)\ne\emptyset)}
\asymp 
|x|^{-\beta_i}
\Cap_{\alpha_i}(A),\quad i=1,\dots,k.
\]
Then for every $|x|$ sufficiently large,
\[
{\mathbb P((\mathcal R_1+\dots+\mathcal R_k)\cap (x+A)\ne\emptyset)}
\asymp
|x|^{(k-1)d-\beta_1-\dots-\beta_k}\Cap_{\alpha_1+\dots+\alpha_k-(k-1)d}(A).
\]
Here $\asymp$ may rely on $d,k,(\alpha_{i}),(\beta_{i})$.
\end{theorem}

\begin{remark}
Take $\alpha_i=\beta_i=d-2$, and $(\mathcal R_i)$ ranges of independent simple random walks, we deduce \cite[Theorem 1.6]{asselah2023branching} for large enough $|x|$.

Take $\alpha_i=d-4,\beta_i=d-2$, by Bai, Delmas and Hu \cite[Theorem 1.1]{bai2024branching}, we can also apply the theorem to ranges of independent critical branching random walks $(\mathcal R_i)$.
\end{remark}

\begin{remark}
To be more precise, we need $|x|$ large enough so that
\[
|x|\ge 2\diam(A)+\max_{1\le i\le k}\1_{\alpha_i\le\beta_i}\diam(A)^{\frac {d-\alpha_i}{d-\beta_i}}\log^{\frac{1}{d-\beta_i}}(\diam(A)+1),
\]
see Section \ref{sec:3.3}.
In particular, if $\alpha_i>\beta_i$, it is enough to have $|x|\ge 2\diam(A)$.
When the sets $(\mathcal R_i)$ are ranges of Markov processes (e.g. stable random walks), one may directly adapt the method of Section \ref{sec:markov} to remove this restriction on $|x|$.
\end{remark}

The main tool we use is fractal percolation, or Mandelbrot percolation: 
\begin{definition}\label{def:1}
Let $d\ge 1$, $p\in[0,1]$.
Consider the natural tiling of $[-\frac 1 2,\frac 1 2)^d$ by $2^d$ cubes congruent to $[0,\frac 1 2)^d$, we keep or discard these cubes independently with probability $p$, and let $Z_1$ be the collection of remaining cubes. 
Then we iterate the process, divide each cube in $Z_k$ into $2^d$ smaller ones congruent to $[0,2^{-k-1})^d$, and keep each of them with probability $p$ independently to obtain $Z_{k+1}$. 
Finally, we call
\[
Q_d(p;k):=\pars*{\bigcup_{Q\in Z_k}2^{k}Q}\cap\mathbb Z^d.
\]
the discrete fractal percolation, and 
\[
Q_d(p):={\bigcap_{n\ge 1}\bigcup_{Q\in Z_n}\overline Q}
\]
the (continuous) fractal percolation.

\begin{figure}[ht]
\centering
\includegraphics[height=4cm]{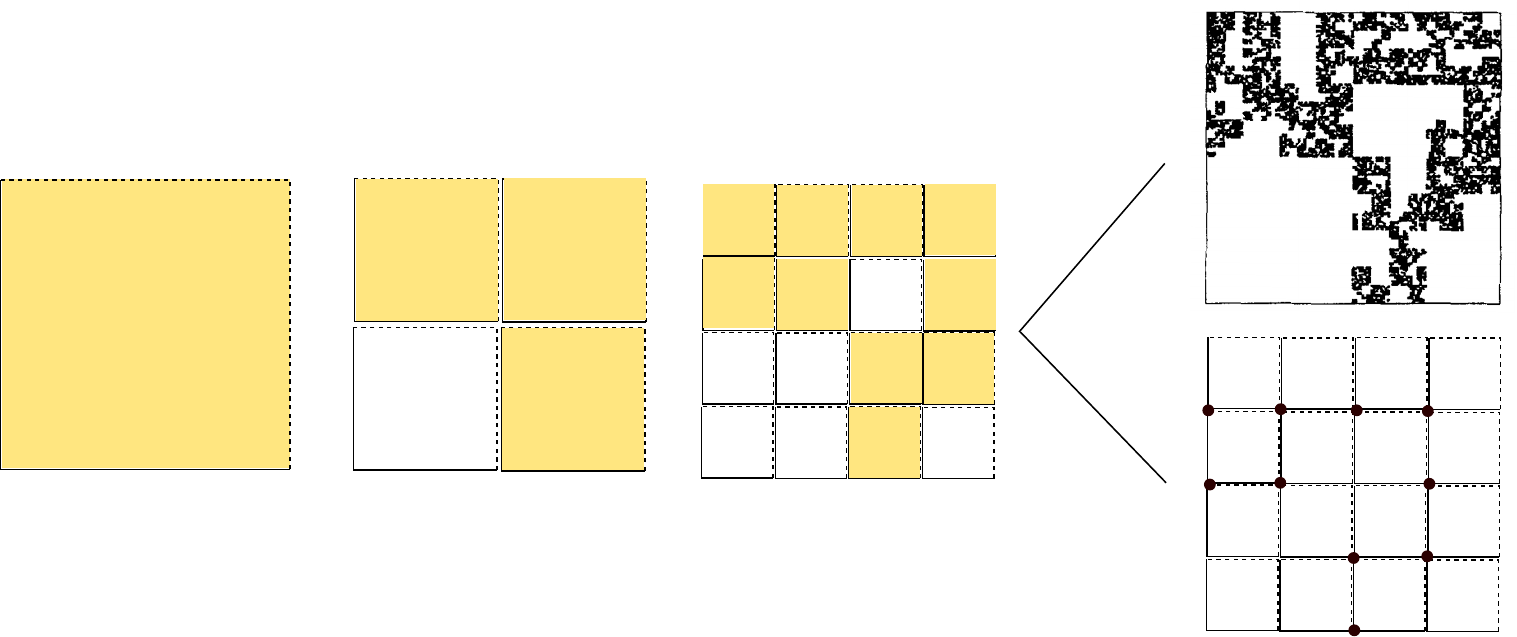}
\caption{Demonstration for construction of fractal percolation. The continuous version is limit of (union of closure of) $Z_k$, and the discrete version is collection of lattice points in the rescaled $2^kZ_k$.}
\end{figure}
\end{definition}

In particular, by Peres \cite[Corollary 4.3]{peres1996intersection}, long-range hitting probabilities of fractal percolation are characterized by capacity in a similar way as \eqref{eq:eg_cap}:
Let $d\ge 1$, $\beta>0$, then for every $k\ge 1$ and $A\subset[-2^{k-1},2^{k-1})^d\cap\mathbb Z^d$,
\begin{equation}\label{eq:FP_cap}
{\mathbb P(Q_d(2^{-\beta};k)\cap A\ne\emptyset)}
\asymp 2^{-\beta k}\Cap_\beta(A),
\end{equation}
where $\asymp$ only depends on $d,\beta$ (independent of $k,A$).

Our key result is the following estimate for sum of fractal percolations:
\begin{theorem}\label{thm:main}
Let $d\ge 1$. Let $p,q\in(0,1)$ such that
\begin{align}\label{eq:nonempty}
\beta:=-\log_2(2^dpq)>0,\quad p,q>2^{-d},
\end{align}
let $Q_d,\widehat Q_d$ be two independent fractal percolations, 
then for every $1\le m\le n$ and every nonempty set 
$A\subset[-2^{n-1},2^{n-1})^d\cap\mathbb Z^d$,
\begin{equation*}
(2^{d}p)^m\lesssim
\frac{\mathbb P((Q_d(p;m)+\widehat Q_d(q;n))\cap A\ne\emptyset)}{q^n\Cap_\beta(A)}
\lesssim (2^dp)^{m\vee\log_2\diam(A)}.
\end{equation*}
Here $\lesssim$ only depends on $d,p,q$ (independent of $m,n,A$).
\end{theorem}

\begin{remark}
As has been observed in \cite{peres1996intersection}, the number of cubes in $(Q_d(p;k))_{k\ge 1}$ forms a branching process, with average offspring $2^dp$. 
When $2^dp>1$, this branching process is supercritical, so each $Q_d(p;k)$, as well as the limit $Q_d(p)$, is nonempty with a positive probability bounded uniformly from below. When $2^dp\le 1$, the branching process is subcritical or critical, and the nonempty probability of $Q_d(p;k)$ decays to $0$ as $k\rightarrow\infty$, making our hitting events trivial. Therefore we need $p,q>2^{-d}$ in \eqref{eq:nonempty}.  
See also Section \ref{sec:log_cap} for a brief discussion on the condition $\beta>0$.
\end{remark}

Except for proving Theorem \ref{thm:rw+rw}, Theorem \ref{thm:main} itself has some interesting consequences. 
Firstly, we may take $m=n$, and consider the continuous limit when $m,n\rightarrow\infty$. 
Here we abuse the notation $\Cap_\beta(\cdot)$ for Newtonian capacity of a compact set $\Lambda\subset\mathbb R^d$, 
\[
\Cap_\beta (\Lambda):=\pars*{\inf_{\mu}\int_{x,y\in \Lambda}|x-y|^{-\beta}\mu(\dif x)\mu(\dif y)}^{-1}.
\]
\begin{corollary}\label{cor:main_continuous}
Let $d\ge 1$, let $p,q\in(0,1)$ satisfy \eqref{eq:nonempty}. For every $n\ge 1$ and every $A\subset[-2^{n-1},2^{n-1})^d\cap\mathbb Z^d$,
\begin{equation}\label{eq:main_discrete+}
{\mathbb P((Q_d(p;n)+\widehat Q_d(q;n))\cap A\ne\emptyset)}
\asymp
{\mathbb P(Q_d(2^dpq;n)\cap A\ne\emptyset)}.
\end{equation}
Moreover, for every closed set $\Lambda\subset[-\frac12,\frac12]^d$,
\begin{align}\label{eq:main_continuous}
{\mathbb P((Q_d(p)+\widehat Q_d(q))\cap\Lambda\ne\emptyset)}
\asymp\Cap_\beta(\Lambda)
\asymp
{\mathbb P(Q_d(2^dpq)\cap\Lambda\ne\emptyset)}.
\end{align}
Here $\asymp$ only depends on $d,p,q$ (independent of $\Lambda$).
\end{corollary}

Corollary \ref{cor:main_continuous} means the random set $Q_d(p)+\widehat Q_d(q)$ is comparable to $Q_d(2^dpq)$, and a phase transition takes place when $2^dpq=1$. This phenomenon has been noticed since Larsson \cite{larsson1991difference}, and has been studied in terms of fractal geometry when $d=1$ in Dekking and Simon \cite{dekking2008size}.

Secondly, we may fix $m$ and let $n\rightarrow\infty$ in Theorem \ref{thm:main} to obtain a comparison between two capacities,
\begin{corollary}\label{cor:cap_a_b}
Let $0<a<b<d$. 
For every nonempty finite set $A\subset\mathbb Z^d$ and every $m\ge 1$,
\begin{align*}
2^{(a-b)(m\vee\log_2\diam(A))}\lesssim
\frac{\mathtt{Cap}_{a}(A)}{\mathbb E[\mathtt{Cap}_b(A+Q_d(2^{b-a-d};m))]}
\lesssim 2^{(a-b)m}.
\end{align*}
For every nonempty closed set $\Lambda\subset[0,1]^d$,
\begin{align*}
\frac{\mathtt{Cap}_{a}(\Lambda)}{\mathbb E[\mathtt{Cap}_b(\Lambda+Q_d(2^{b-a-d}))]}\asymp 1.
\end{align*}
Here $\lesssim$ and $\asymp$ may rely on $d,a,b$.
\end{corollary}

Studies of fractal percolation date back to Mandelbrot \cite{mandelbrot1972renewal}.
In particular, its relation to capacity in terms of hitting probability was discovered in \cite{peres1996intersection}.
Minkowski sum of sets is equivalent to algebraic difference,
\[
A-B:=\{a-b,a\in A,b\in B\},
\]
and this structure for Cantor sets was first introduced in Palis \cite{palis58homoclinic}, and adapted to random cases in \cite{larsson1991difference}.
The sum of fractal percolations in dimension $d=1$ (Mandelbrot cascades) was considered in \cite{dekking2008size}.
More recently, sum of simple random walks was studied in \cite{asselah2023branching}. Notably, fractal percolation methods can also be used to give intersection probability estimates for random walks (Lawler \cite{lawler1982probability}), branching random walks (Baran \cite{baran2024intersections}), as well as in continuous regime for Brownian motions and Brownian sheets (Khoshnevisan and Shi \cite{MR1733143}).

The paper is outlined as follows. In Section \ref{sec:main} we prove Theorem \ref{thm:main} in a few steps: 
\begin{itemize}
\item We first introduce in Section \ref{sec:percolation} that,  fractal percolation is equivalent to percolation on a tree.
\item Then we show in Section \ref{sec:lower} that lower bound of  Theorem \ref{thm:main} follows by a standard second moment method.
\item In Section \ref{sec:markov}, we observe that sum of fractal percolation can be described using product of two Markov chains, and one can apply a method first introduced in Fitzsimmons and Salisbury \cite{fitzsimmons1989capacity} and Salisbury \cite{MR1395618}, later simplified in Lyons, Peres and Schramm \cite{lyons2003markov} to study product of Markov chains.
\item Finally we prove the upper bound of Theorem \ref{thm:main} in Section \ref{sec:upper}.
\end{itemize}
In Section \ref{sec:3.1}-\ref{sec:3.3}, we prove Corollary \ref{cor:main_continuous}, Corollary \ref{cor:cap_a_b} and Theorem \ref{thm:rw+rw} respectively, and we end the paper in Section \ref{sec:log_cap} by a brief discussion on sum of fractal percolations with parameters not covered by \eqref{eq:nonempty}.

\section{Proof of Theorem \ref{thm:main}}\label{sec:main}

For each $k\ge 1$, we abbreviate the centered discrete cube with side length $2^k$ by 
\[
\Delta_k:=[-2^{k-1},2^{k-1})^d\cap\mathbb Z^d.
\]
\subsection{Percolation on tree}\label{sec:percolation}
\begin{figure}[ht]
\centering
\includegraphics[height=5cm]{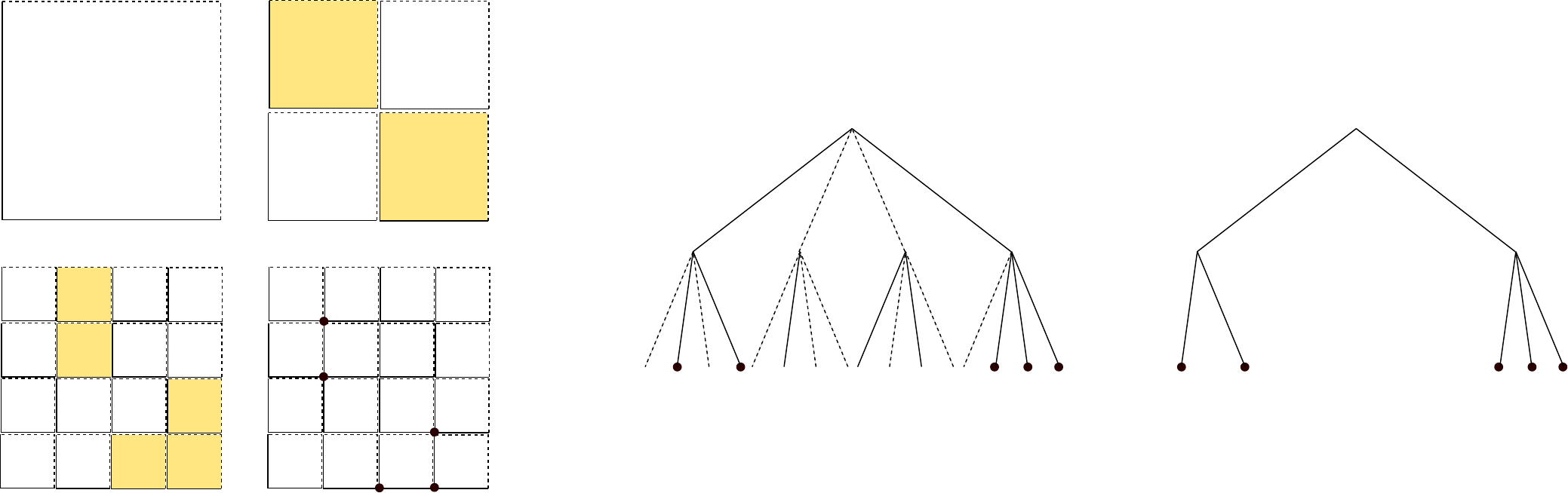}
\caption{An illustration of $Q_k(p)$ and its corresponding percolation on $\Gamma^{(k)}$.}
\label{fig:1}
\end{figure}

Our first observation is that (discrete) fractal percolation $Q_d(p;k)$ on $\Delta_k$ can be described by the Bernoulli bond percolation on a (finite) $2^d$-ary tree defined as follows, see Figure \ref{fig:1} for demonstration. 

\begin{definition}
Let $\Xi=\{0,1\}^d$ be our alphabet (heuristically, we use $\Xi$ to represent the $2^d$ small cubes when cutting one cube into half size). For $j\ge 1$, let $\Xi^j$ denote words $(\theta_1,\dots,\theta_j)$ with $\theta_1,\dots,\theta_j\in\Xi$. By convention, let $\Xi^0:=\{\emptyset\}$, let
\[
\mathbf 0_d=(0,\dots,0), \mathbf 1_d=(1,\dots,1)\in\Xi.
\]

Denote by $\Gamma^{(k)}$ a $2^d$-ary tree of height $k$ with vertices labelled by $\Xi^0\cup\dots\cup \Xi^k$: the tree is rooted at $\emptyset\in \Xi^0$, and for each $j=1,2,\dots,k$, generation $j$ of the tree is given by $\Xi^j$, with an edge connecting $(\theta_1,\dots,\theta_j)\in\Xi^j$ to its parent $(\theta_1,\dots,\theta_{j-1})$.
\end{definition}

For each particle $(\theta_1,\dots,\theta_j)\in\Xi^j$, let
\[
\mathscr R_k(\theta_1,\dots,\theta_j):=-2^{k-1}\mathbf 1_d+\sum_{i=1}^j 2^{k-i}\theta_i+\Delta_{k-j}\subset\mathbb Z^d.
\]
This gives a bijection between $\Xi^j$ and the natural tiling of $\Delta_k$ by blocks of size $\Delta_{k-j}$.

On $\Gamma^{(k)}$, we perform Bernoulli bond percolation with parameter $p$, that is, we independently keep each edge open with probability $p$, closed with probability $1-p$, and let
\[
O_j:=\setof{\theta=(\theta_1,\dots,\theta_j)}{\text{there is an open path between }\emptyset \text{ and }\theta}.
\]
Under a suitable coupling probability measure, $\mathscr R_k$ gives a bijection between $O_j$ and $2^kZ_j$, where we recall $Z_j$ from Definition \ref{def:1}.

In other words, we may take a coupling probability measure $\mathbb P$ under which 
\[
O_j=\mathscr R_k^{-1}(2^kZ_j),\,j=1,2,\dots,k.
\]
In particular, when $j=k$,
\begin{align}\label{eq:iff1}
\text{there is an open path between $\emptyset$ and $(\theta_1,\dots,\theta_k)$}
\iff
\mathscr R_k((\theta_1,\dots,\theta_k))\in Q_d(p;k).
\end{align}

As a direct application, we have the following estimates for fractal percolation:
\begin{lemma}\label{lem:xy_distance}
For any $x,y\in\Delta_k$,
\[
\mathbb P(x\in Q_d(p;k))=p^{k},
\]
\[
\mathbb P(x,y\in Q_d(p;k))\le p^{k+\log_2|x-y|-\frac 1 2\log_2 d}.
\]
\end{lemma}
\begin{proof}
The first equation is clear, because $x\in Q_d(p;k)$ means all the $k$ edges on the path from $\emptyset$ to $\mathscr R_k^{-1}(x)$ are open, which happens with probability $p^k$.

For the second, denote by $\theta$ the common ancestor of $\mathscr R_k^{-1}(x)$ and $\mathscr R_k^{-1}(y)$, and by 
$h$ the length from $\theta$ to $\mathscr R_k^{-1}(x)$, then it is immediate that  
\[
\mathbb P(x,y\in Q_d(p;k))=p^{k+h}.
\]
Moreover, $\mathscr R_k(\theta)$ is congruent to $\Delta_h$, so the distance between $x,y\in \mathscr R_k(\theta)$ is bounded by
\begin{align}\label{eq:h_bound}
|x-y|\le \sqrt d 2^h,
\end{align}
as shown in Figure \ref{fig:xy}.
\begin{figure}[ht]
\centering
\includegraphics[height=5cm]{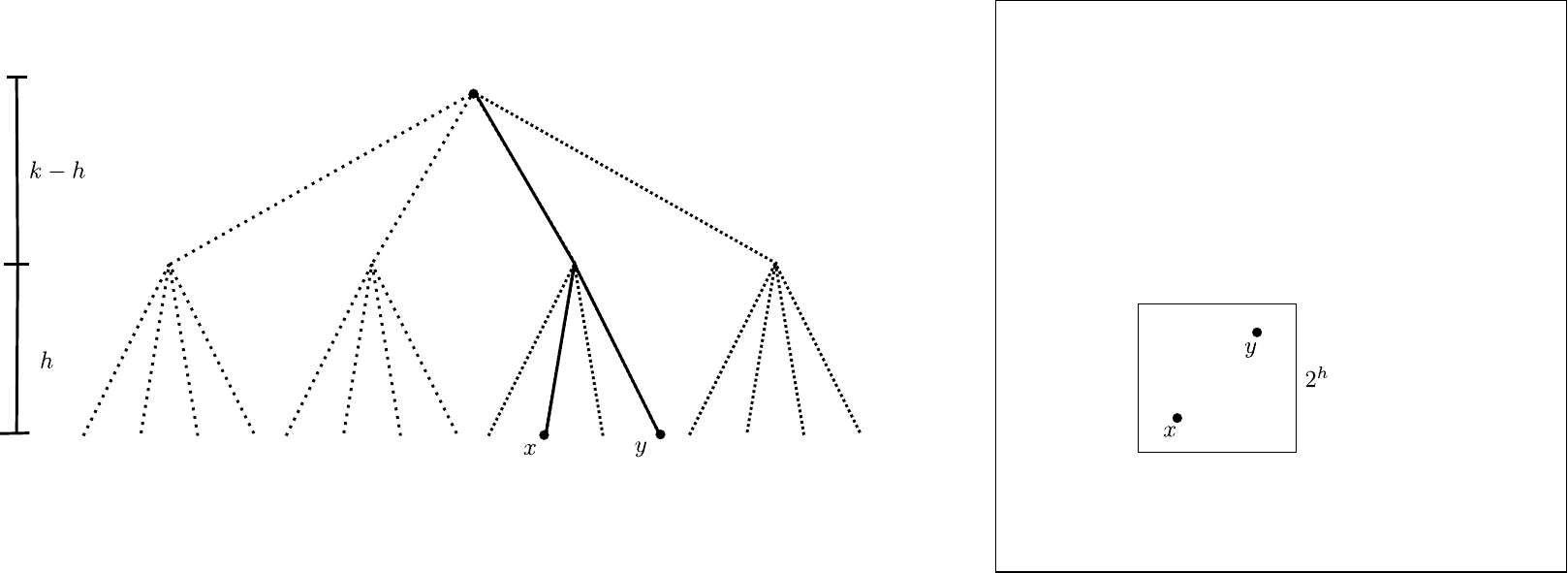}
\caption{Illustration for Lemma \ref{lem:xy_distance}. On the tree $\Gamma^{(k)}$, two points corresponding to $x,y$ has common ancestor of height $h$. In $\mathbb Z^d$, $x,y$ fall into a cube congruent to $\Delta_h$.}
\label{fig:xy}
\end{figure}
\end{proof}

\subsection{Lower bound of Theorem \ref{thm:main}}
\label{sec:lower}
We are now ready to give the lower bound of Theorem \ref{thm:main} by a standard second moment method. For later usage, we slightly extend the domain where the target set stays.

\begin{lemma}\label{lem:main_lower}
Let $d\ge 1$, $\lambda\in[0,1)$. Let $p,q\in(0,1)$ with property \eqref{eq:nonempty}.
For every $1\le m\le n$ and every finite set 
$A\subset[-2^{n-1}-\lambda2^{m-1},2^{n-1}+\lambda2^{m-1})^d\cap\mathbb Z^d$,
\begin{equation*}
\mathbb P((Q_d(p;m)+\widehat Q_d(q;n))\cap A\ne\emptyset)
\gtrsim (2^{d}p)^mq^n\Cap_\beta(A).
\end{equation*}
Here $\gtrsim$ relies on $d,\lambda,p,q$.
\end{lemma}

\begin{proof}
By definition of $\Cap_\beta$ in \eqref{eq:def_cap}, 
it suffices to show that for any probability measure $\mu$ supported on $A$,
\begin{align}
\mathbb P((Q_d(p;m)+\widehat Q_d(q;n))\cap A\ne\emptyset)
\gtrsim (2^{d}p)^mq^n\pars*{\sum_{a,b\in A}\mu(a)\mu(b)|a-b|^{-\beta}}^{-1}.
\label{eq:goal_energy}
\end{align}

Arbitrarily fix $\mu$, let 
\[
L_\mu:=\sum_{a\in A}\mu(a)\sum_{x\in\mathbb Z^d}\1_{x\in Q_d(p;m),a-x\in \widehat Q_d(q;n)}.
\]
By Paley-Zygmund inequality,
\begin{equation}\label{eq:PZ}
\mathbb P((Q_d(p;m)+\widehat Q_d(q;n))\cap A\ne\emptyset)
\ge\mathbb P(L_\mu>0)\ge\frac{(\mathbb E[L_\mu])^2}{\mathbb E[L_\mu^2]}.
\end{equation}

Observe that uniformly in $a\in A$, (here $\asymp$ depends on $\lambda$, in particular it requires the condition $\lambda<1$)
\begin{equation}\label{eq:avoid_boundary}
\#\setof*{x\in\mathbb Z^d}{x\in\Delta_m, 
a-x\in\Delta_n}\asymp 2^{dm}.
\end{equation}
Therefore by Lemma \ref{lem:xy_distance},
\begin{align}\label{eq:first_moment}
\mathbb E[L_\mu]=\sum_{a\in A}\mu(a)\sum_{x\in\Delta_m}\mathbb P(x\in Q_d(p;m))\mathbb P(a-x\in\widehat Q_d(q;n))\gtrsim 2^{dm}p^{m}q^n.
\end{align}

Moreover, apply Lemma \ref{lem:xy_distance} again, we have
\begin{align*}
\mathbb E[L_\mu^2]=&\sum_{a,b\in A}\mu(a)\mu(b)\sum_{x,y\in\mathbb Z^d}
\mathbb P(x,y\in Q_d(p;m))\mathbb P(a-x,b-y\in \widehat Q_d(q;n))
\\
\lesssim&
\sum_{a,b\in A}\mu(a)\mu(b)\sum_{x,y\in\mathbb Z^d}
p^{m+\log_2|x-y|}q^{n+\log_2|a-x+y-b|}\1_{x,y\in\Delta_m}\1_{a-x,b-y\in\Delta_n}
\\
\lesssim&
2^{dm}p^mq^n\sum_{a,b\in A}\mu(a)\mu(b)\sum_{z\in\Delta_{m+1}}
p^{\log_2|z|}q^{\log_2|a-b-z|},
\end{align*}
where in the last step we change variable $z=x-y\in\Delta_{m+1}$.

By \eqref{eq:nonempty}, we can expand the sum on $z$ and continue the calculation above,
\begin{equation}\label{eq:second_moment_L}
\begin{aligned}
\mathbb E[L_\mu^2]\lesssim&
2^{dm}p^mq^n\sum_{a,b\in A}\mu(a)\mu(b)\sum_{z\in\Delta_{m+1}}
p^{\log_2|z|}q^{\log_2|a-b-z|}\\
\lesssim&
2^{dm}p^mq^n\sum_{a,b\in A}\mu(a)\mu(b)
(2^dp)^{(m\wedge\log_2|a-b|)}q^{\log_2|a-b|}\1_{a\ne b}
\\
\lesssim&
2^{dm}p^mq^n\sum_{a,b\in A}\mu(a)\mu(b)|a-b|^{-\beta}\1_{a\ne b}.
\end{aligned}
\end{equation}

Combine \eqref{eq:PZ}, \eqref{eq:first_moment} and \eqref{eq:second_moment_L}, and the conclusion follows.
\end{proof}

\subsection{Relation with Markov chain}\label{sec:markov}
The percolation on $\Gamma^{(k)}$ in Section \ref{sec:percolation} can be identified by the following sequence $(Y_i^{(k)})_{i=1,2,\dots,2^{dk}}$. 
\begin{definition}\label{def:tree_to_Markov}
Fix a percolation on $\Gamma^{(k)}$. We list all the $2^{dk}$ elements of $\Xi^k\subset\Gamma^{(k)}$ of the tree in lexicographical order,
\begin{align}\label{eq:label_vertices}
u_1=(\mathbf 0_d,\dots,\mathbf 0_d),u_2,\dots,u_{2^{dk}}=(\mathbf 1_d,\dots,\mathbf 1_d)\in\Xi^k.
\end{align}

For each $1\le i\le 2^{dk}$, the path from $\emptyset$ to $u_i$ on the tree has length $k$, and we define $Y^{(k)}_i=(y_{i1},\dots,y_{ik})\in\{0,1\}^k$, where
\[
y_{ij}:=\1_{\braces{\text{The $j$-th step in the path from $\emptyset$ to $u_i$ is open}}}.
\]
\end{definition}
In particular, 
\begin{align}\label{eq:iff2}
Y^{(k)}_i=(1,1,\dots,1) \iff
\text{There is an open path between $\emptyset$ and $u_i$}.
\end{align}
\begin{figure}[ht]
\centering
\includegraphics[height=5cm]{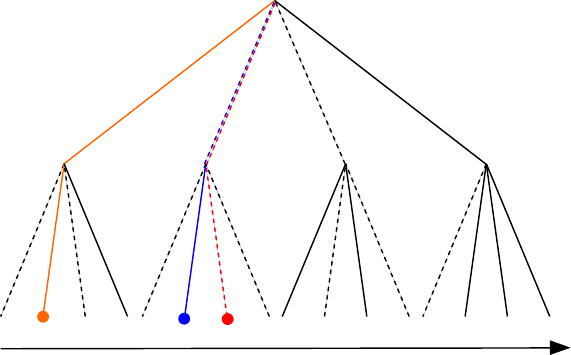}
\caption{In this picture we list $(u_i)$ for $\Gamma^{(2)}$ from left to right. Here the colored paths correspond to $Y^{(2)}_2=(1,1), Y^{(2)}_6=(0,1), Y^{(2)}_{7}=(0,0)$.}
\end{figure}

By construction, it is immediate that
\begin{lemma}\label{lem:chain}
Let $\mathbb P$ be the law under which we have Bernoulli bound percolation on $\Gamma^{(k)}$ with parameter $p$.
Under $\mathbb P$, $(Y_i^{(k)})_i$ is a Markov chain.
\end{lemma}

Finally, we cite the following theorem on general Markov chains from \cite{lyons2003markov}, which is a simplified version of the method first introduced in \cite{fitzsimmons1989capacity},see also \cite[Theorem 2]{MR1395618}, and shall provide an important bound on hitting probability estimates of fractal percolation.
\begin{theorem}\label{thm:salisbury}
Let $(X_n)_{n\in\mathbb N}$, $(Y_n)_{n\in\mathbb N}$ be two independent Markov chains with the same state space $\mathcal S$. 
For any set 
\[
K\subset
\setof{(i,j;x,y)}{i,j\in\mathbb N,x,y\in\mathcal S},
\] 
if
\[
\mathbb P\pars*{\exists i,j,
(i,j;X_i,Y_j)\in K
}>0,
\]
then there are explicit nonnegative coefficients $(a_{ijxy})$ such that 
\[
\frac{\mathbb E[S_K]^2}{\mathbb E[S_K^2]}
\le
\mathbb P\pars*{\exists i,j,
(i,j;X_i,Y_j)\in K
}
\le
64\frac{\mathbb E[S_K]^2}{\mathbb E[S_K^2]},
\]
where 
\[
S_K:=
\sum_{(i,j;x,y)\in K}
a_{ijxy}\1_{X_i=x,Y_j=y},
\]
\end{theorem}
\begin{proof}
The theorem is \cite[Theorem 3.1]{lyons2003markov}, and the explicit construction for $S_K$ is given in \cite[Page 7]{lyons2003markov} by
\[
a_{ijxy}=\mathbb P\parsof{\tau=i,\lambda=j}{X_i=x,Y_j=y},
\]
where 
\[
\tau:=\min\setof{i}{\exists j,(i,j;X_i,Y_j)\in K},
\]
\[
\lambda:=\min\setof{j}{(\tau,j;X_\tau,Y_j)\in K}.
\]
\end{proof}

\subsection{Upper bound of Theorem \ref{thm:main}}
\label{sec:upper}
For simplicity, we shall denote
\[
i\oplus j:=\mathscr R_m(u_i)+\mathscr R_n(\widehat u_j)\in\mathbb Z^d
\]
for every $1\le i\le 2^{dm}, 1\le j\le 2^{dn}$, where $(u_i)$ and $(\widehat u_j)$ enumerate the elemements in the last generation of $\Gamma^{(m)}$ and
$\Gamma^{(n)}$, as in Definition \ref{def:tree_to_Markov}. 
We also denote by $h^m_{ii'}$ (similarly for $h^n_{jj'}$) the graph distance from $u_i$ to common ancestor of $u_i$ and $u_{i'}$ in $\Gamma^{(m)}$, as shown in Figure \ref{fig:ancestor}.
\begin{figure}[ht]
\centering
\includegraphics[height=5cm]{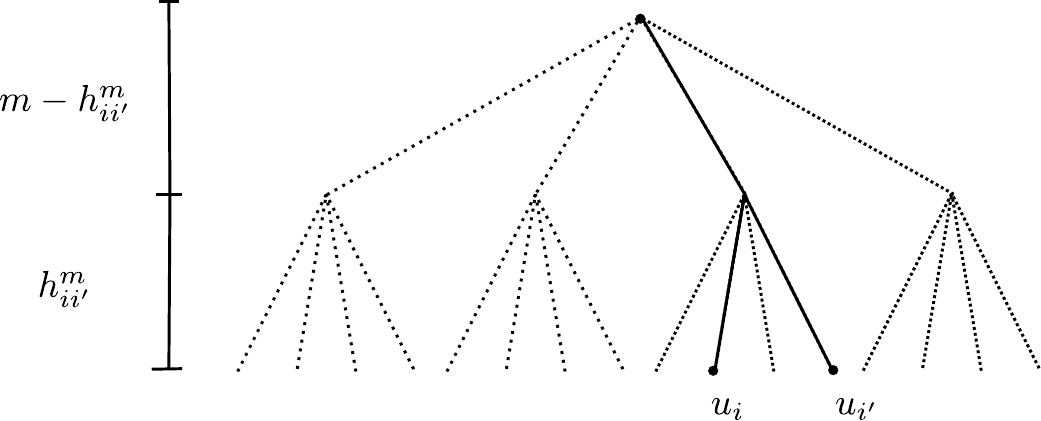}
\caption{Illustration for $h^m_{ii'}$, $u_i$, and $u_{i'}$.}
\label{fig:ancestor}
\end{figure}

\begin{lemma}\label{lem:up1}
Let $d\ge 1$, let $p,q\in(0,1)$ with property \eqref{eq:nonempty}. For every $1\le m\le n$ and every finite set 
$A\subset[-2^{n-1}-2^{m-1},2^{n-1}+2^{m-1})^d\cap\mathbb Z^d$,
there is a family of nonnegative coefficients $(a_{ij})_{1\le i\le 2^{dm},1\le j\le 2^{dn}}$ such that
\begin{equation}\label{eq:aij0}
\begin{aligned}
a_{ij}=0\text{ when }i\oplus j\not\in A,\quad
\sum_{i,j}a_{ij}=1,
\end{aligned}
\end{equation}
and
\[
{\mathbb P\pars*{
(Q_d(p;m)+\widehat Q_d(q;n))\cap A\ne\emptyset
}}
\le
64p^mq^n\pars*{\sum_{i,i',j,j'}a_{ij}a_{i'j'}
p^{h^m_{ii'}}q^{h^n_{jj'}}}^{-1}.
\]
\end{lemma}
\begin{proof}
Let $Y$ and $\widehat Y$ be the Markov chains of Lemma \ref{lem:chain} associated with $Q_d(p;m)$ and $\widehat Q_d(q;n)$.
Abbreviate $
\mathbf 1_k:=(1,1,\dots,1)$ for $k$ consecutive $1$s.

We apply Theorem \ref{thm:salisbury} to them together with the set
\[
K:=\setof{(i,j;\mathbf 1_m,\mathbf 1_n)}{i\oplus j\in A},
\]
then there are nonnegative coefficients $(a_{ij})$ such that
\[
\mathbb P\pars*{\exists i,j,
Y_i=\mathbf 1_m,\widehat Y_j=\mathbf 1_n,i\oplus j\in A
}
\le
64\frac{\mathbb E[S_A]^2}{\mathbb E[S_A^2]},
\]
where \[
S_A=\sum_{i,j}a_{ij}\1_{Y_i=\mathbf 1_m,\widehat Y_j=\mathbf 1_n}\1_{i\oplus j\in A}.
\]

By \eqref{eq:iff1} and \eqref{eq:iff2}, 
\[
Y_i=\mathbf 1_m,\widehat Y_j=\mathbf 1_n\iff\mathscr R_m(u_i)\in Q_d(p;m),\mathscr R_n(\widehat u_j)\in \widehat Q_d(q;n),
\]
so we obtain
\begin{align}\label{eq:bound_second_moment}
\mathbb P\pars*{
(Q_d(p;m)+\widehat Q_d(q;n))\cap A\ne\emptyset
}
\le
64\frac{\mathbb E[S_A]^2}{\mathbb E[S_A^2]}.
\end{align}

Note that $\frac{\mathbb E[S_A]^2}{\mathbb E[S_A^2]}$ does not change if we replace $(a_{ij})$ by its multiple $(ca_{ij})$, and does not change if we change $a_{ij}$ for some index $i\oplus j\not\in A$,
we can modify these coefficients to fit \eqref{eq:aij0}.

Finally, we expand
\[
\mathbb E[S_A]=\sum_{i,j}a_{ij}\mathbb P(Y_i=\mathbf 1_m,\widehat Y_j=\mathbf 1_n)=p^{m}q^n,
\]
and
\begin{align*}
\mathbb E[S_A^2]
=&\sum_{i,i',j,j'}a_{ij}a_{i'j'}
\mathbb P(Y_i=Y_{i'}=\mathbf 1_m)
\mathbb P(\widehat Y_j=\widehat Y_{j'}=\mathbf 1_n)\\
=&\sum_{i,i',j,j'}a_{ij}a_{i'j'}
p^{m+h^m_{ii'}}q^{n+h^n_{jj'}}.
\end{align*}
The conclusion follows by putting these equations into \eqref{eq:bound_second_moment}. 
\end{proof}

\begin{lemma}\label{lem:up2}
Let $d\ge 1$, let $p,q\in(0,1)$ with property \eqref{eq:nonempty}.
Let $1\le m\le n$ and $A\subset[-2^{n-1}-2^{m-1},2^{n-1}+2^{m-1})^d\subset\mathbb Z^d$.
Let $(a_{ij})$ be any collection of nonnegative coefficients satisfying \eqref{eq:aij0}, then
\[
(2^dp)^{\log_2\diam(A)\vee m}\sum_{i,i',j,j'}a_{ij}a_{i'j'}
p^{h^m_{ii'}}q^{h^n_{jj'}}
\gtrsim
p^{m}\sum_{i,i',j,j'}a_{ij}a_{i'j'}|i\oplus j-i'\oplus j'|^{-\beta}.
\]
Here $\gtrsim$ only depend on $d,p,q$ (independent of $m,n,A$).
\end{lemma}
\begin{proof}

We claim that
\begin{enumerate}
\item
\begin{equation}\label{eq:good_side}
\begin{aligned}
&2^{dm}\sum_{i,i',j,j'}a_{ij}a_{i'j'}
p^{h^m_{ii'}}q^{h^n_{jj'}}\\
\gtrsim&\sum_{i,i',j,j':|\mathscr R_n(\widehat u_j)-\mathscr R_n(\widehat u_{j'})|\le 2\sqrt d 2^m}a_{ij}a_{i'j'}|i\oplus j-i'\oplus j'|^{-\beta}.
\end{aligned}
\end{equation}
\item If $m<\log_2\diam(A)$, then
\begin{equation}\label{eq:bad_side}
\begin{aligned}
&(2^dp)^{\log_2\diam(A)}\sum_{i,i',j,j'}a_{ij}a_{i'j'}
p^{h^m_{ii'}}q^{h^n_{jj'}}\\
\gtrsim&p^m\sum_{i,i',j,j':|\mathscr R_n(\widehat u_j)-\mathscr R_n(\widehat u_{j'})|\ge 2\sqrt d 2^m}a_{ij}a_{i'j'}|i\oplus j-i'\oplus j'|^{-\beta},
\end{aligned}
\end{equation}
\end{enumerate}

Let us first demonstrate that these claims are sufficient to give Lemma \ref{lem:up2}.
In fact, if $m\ge\log_2\diam(A)$, then 
the condition $|\mathscr R_n(\widehat u_j)-\mathscr R_n(\widehat u_{j'})|\le 2\sqrt d2^m$ trivially includes every choice of $(\widehat u_j,\widehat u_{j'})$ such that $i\oplus j,i'\oplus j'\in A$, because
\begin{align*}
|\mathscr R_n(\widehat u_j)-\mathscr R_n(\widehat u_{j'})|
\le &|i\oplus j-i'\oplus j'|+|\mathscr R_m(u_i)-\mathscr R_m(u_i')|\\
\le &\diam(A)+\sqrt d2^m\le 2\sqrt d2^m.
\end{align*}
Therefore, \eqref{eq:good_side} is already our desired result. If $m<\log_2\diam(A)$, then $(2^dp)^m<(2^dp)^{\log_2\diam(A)}$ (recall by \eqref{eq:nonempty} that $2^dp>1$), and the conclusion follows by adding \eqref{eq:good_side} and \eqref{eq:bad_side}.

Let us begin by proving \eqref{eq:bad_side}, given that $m<\log_2\diam(A)$.
By the trivial bound $p^{h^m_{ii'}}\ge p^m$, it suffices to show that
\begin{equation}\label{eq:bad_sidea}
\begin{aligned}
&(2^{d}p)^{\log_2\diam(A)}\sum_{i,i',j,j'}a_{ij}a_{i'j'}
q^{h^n_{jj'}}\\
\gtrsim&\sum_{i,i',j,j':|\mathscr R_n(\widehat u_j)-\mathscr R_n(\widehat u_{j'})|\ge 2\sqrt d 2^m}a_{ij}a_{i'j'}|i\oplus j-i'\oplus j'|^{-\beta},
\end{aligned}
\end{equation}
where we recall that $\beta=-\log_2(2^dpq)>0$.
Since $(a_{ij})$ is supported on those $(i,j)$ with $i\oplus j\in A$, and $2^dp>1$, we can simplify the right hand side by
\begin{align*}
&\text{RHS}_{\eqref{eq:bad_sidea}}\\
=&\sum_{i,i',j,j':|\mathscr R_n(\widehat u_j)-\mathscr R_n(\widehat u_{j'})|\ge 2\sqrt d 2^m}a_{ij}a_{i'j'}(2^dpq)^{\log_2|i\oplus j-i'\oplus j'|}\\
\le&(2^dp)^{\log_2\diam(A)}\sum_{i,i',j,j':|\mathscr R_n(\widehat u_j)-\mathscr R_n(\widehat u_{j'})|\ge 2\sqrt d 2^m}a_{ij}a_{i'j'}q^{\log_2|i\oplus j-i'\oplus j'|}.
\end{align*}
Note that for each $i,i'$, $|\mathscr R_m(u_i)-\mathscr R_m(u_{i'})|\le\sqrt d 2^m$ because they live in $\Delta_m$. Therefore, $|\mathscr R_n(\widehat u_j)-\mathscr R_n(\widehat u_{j'})|\ge 2\sqrt d 2^m$ implies
\[
|i\oplus j-i'\oplus j'|\ge \frac 12|\mathscr R_n(\widehat u_j)-\mathscr R_n(\widehat u_{j'})|,
\]
so we further simplify
\begin{align*}
&\text{RHS}_{\eqref{eq:bad_sidea}}\\
\lesssim &(2^dp)^{\log_2\diam(A)}\sum_{i,i',j,j':|\mathscr R_n(\widehat u_j)-\mathscr R_n(\widehat u_{j'})|\ge 2\sqrt d 2^m}a_{ij}a_{i'j'}q^{\log_2|\mathscr R_n(\widehat u_j)-\mathscr R_n(\widehat u_{j'})|}\\
\le &(2^dp)^{\log_2\diam(A)}\sum_{i,i',j,j'}a_{ij}a_{i'j'}{|\mathscr R_n(\widehat u_j)-\mathscr R_n(\widehat u_{j'})|}^{\log_2 q}.
\end{align*}
Put this into \eqref{eq:bad_sidea}, it now suffices to show 
\begin{align*}
\sum_{i,i',j,j'}a_{ij}a_{i'j'}
q^{h^n_{jj'}}
\gtrsim \sum_{i,i',j,j'}a_{ij}a_{i'j'}{|\mathscr R_n(\widehat u_j)-\mathscr R_n(\widehat u_{j'})|}^{\log_2 q},
\end{align*}
and this is known by \cite[Theorem 4.2]{peres1996intersection}.
Hence we have proved \eqref{eq:bad_side}.

It remains to prove \eqref{eq:good_side}.
Denote by $|v|$ the generation (distance to the root) of a vertex $v$ in a tree, and denote
\[
a_{[v][\widehat{v}]}=\sum_{i\succeq v,j\succeq \widehat{v}}a_{ij},
\]
then we can bound the sum in left-hand side of \eqref{eq:good_side} by
\begin{equation}\label{eq:R0}
\begin{aligned}
2^{-dm}\text{LHS}_{\eqref{eq:good_side}}=&\sum_{i,i',j,j'}a_{ij}a_{i'j'}
p^{h^m_{ii'}}q^{h^n_{jj'}}\\
\gtrsim&\sum_{v\in\Gamma^{(m)},\widehat{v}\in\Gamma^{(n)}}p^{m-|v|}q^{n-|\widehat{v}|}\sum_{i,i'\succeq v,j,j'\succeq \widehat{v}}a_{ij}a_{i'j'}\\
\gtrsim&\sum_{h=0}^m\sum_{|v|=m-h,|\widehat{v}|=n-h}p^{h}q^{h}\sum_{i,i'\succeq v,j,j'\succeq \widehat{v}}a_{ij}a_{i'j'}\\
=&\sum_{h=0}^m\sum_{|v|=m-h,|\widehat{v}|=n-h}p^{h}q^{h}a_{[v][\widehat{v}]}^2,
\end{aligned}
\end{equation}
where in the second line we are summing over each vertex $v$ (and $\widehat v$) on the path from root to common ancestor of $u_i,u_{i'}$ (and $\widehat u_j,\widehat u_{j'}$).

For each $h=0,1,\dots,m$, we say four vertices $(x,y,z,w)$ are $h$-regular, if $x,y$ are in the $(m-h)$-th generation of $\Gamma^{(m)}$, 
$z,w$ are in the $(n-h)$-th generation of $\Gamma^{(n)}$, and there exists a quadruple $(i,i',j,j')$ such that $(x,y,z,w)$ are ancestors of $(u_i,u_{i'},u_j,u_{j'})$ respectively (denoted by 
$
(x,y,z,w)\preceq(i,i',j,j')
$),
and
\begin{align}\label{eq:def_Xi}
|i\oplus j-i'\oplus j'|\le 3{\sqrt d}2^{h}.
\end{align}
Denote by $\Theta_h$ the collection of all $h$-regular quadruples, as illustrated in Figure \ref{fig:Xi}.

\begin{figure}[ht]
\centering
\includegraphics[height=5cm]{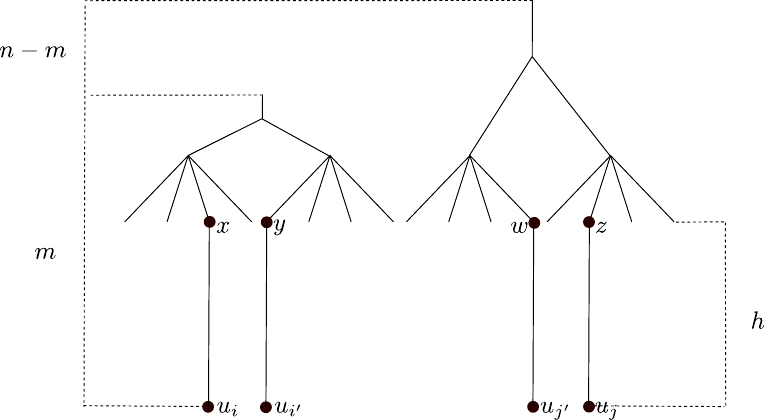}
\caption{Illustration of $(x,y,z,w)\in\Theta_h$. 
}
\label{fig:Xi}
\end{figure}

We observe from the definition of $\Theta_h$ that,
\begin{align}
&\sum_{h=0}^m\sum_{(x,y,z,w)\in\Theta_h}
\sum_{(i,i',j,j')\succeq(x,y,z,w)}2^{-\beta h}a_{ij}a_{i'j'}\nonumber\\
\ge&\sum_{h=0}^m\sum_{(x,y,z,w)\in\Theta_h}
\sum_{(i,i',j,j')\succeq(x,y,z,w)}2^{-\beta h}a_{ij}a_{i'j'}\1_{\braces{|i\oplus j-i'\oplus j'|\le 3{\sqrt d}2^{h}}}\nonumber\\
=&\sum_{i,i',j,j'}a_{ij}a_{i'j'}\sum_{h=0}^m 2^{-\beta h}\1_{\braces{|i\oplus j-i'\oplus j'|\le 3{\sqrt d}2^{h}}}\nonumber\\
\gtrsim &\sum_{i,i',j,j':|i\oplus j-i'\oplus j'|\le 3\sqrt d 2^m}a_{ij}a_{i'j'}\cdot \pars{|i\oplus j-i'\oplus j'|}^{-\beta}\gtrsim \text{RHS}_{\eqref{eq:good_side}}.\label{eq:R1}
\end{align}

Moreover,
\begin{align}
&\sum_{h=0}^m\sum_{(x,y,z,w)\in\Theta_h}
\sum_{(i,i',j,j')\succeq(x,y,z,w)}2^{-\beta h}a_{ij}a_{i'j'}
\nonumber\\
=&\sum_{h=0}^m\sum_{(x,y,z,w)\in\Theta_h}
2^{-\beta h}a_{[x][z]}a_{[y][w]}
\nonumber\\
\le&\sum_{h=0}^m2^{-\beta h}
\sum_{(x,y,z,w)\in\Theta_h}
\frac 1 2(a^2_{[x][z]}+a^2_{[y][w]})
\nonumber\\
=&\sum_{h=0}^m2^{-\beta h}\sum_{x,z}
a^2_{[x][z]}\#\setof{(y,w)}{(x,y,z,w)\in\Theta_h},
\label{eq:xyzw}
\end{align}
where in the last line we use that $(x,z)$ and $(y,w)$ are symmetric.

For each $x\in\Gamma^{(k)}$, denote by $\mathscr R_k^*(x)$ the center of the cube $\mathscr R_k(x)$. When $x$ is ancestor of $u_i$ in generation $h$, we have
\[
|\mathscr R^*_k(x)-\mathscr R_k(u_i)|\le \sqrt d 2^{h}.
\]
So for every $(x,y,z,w)\in\Theta_h$,
\begin{align*}
|\mathscr R^*_k(x)-\mathscr R^*_k(y)+\mathscr R^*_k(z)-\mathscr R^*_k(w)|\le 7\sqrt d 2^{h}.
\end{align*}
In particular, for each triple $(x,y,z)$, if we want to choose $w$ so that $(x,y,z,w)\in\Theta_h$, then
$\mathscr R^*_k(w)$ lives in a ball of side length $7\sqrt d 2^{h}$. For different choices of $w$, $(\mathscr R_k(w))_w$ are disjoint cubes of side length $2^{h}$, so we have at most a constant number ($(2+7\sqrt d)^d$) of choices for it, as illustrated in Figure \ref{fig:2}. Therefore, for each $(x,z)$, by choosing $y$ arbitrarily and then using the observation above, we obtain
\begin{align}\label{eq:ancestor_close}
\#\setof{(y,w)}{(x,y,z,w)\in\Theta_h}\lesssim 2^{d(m-h)}.
\end{align}

\begin{figure}[ht]
\centering
\includegraphics[height=5cm]{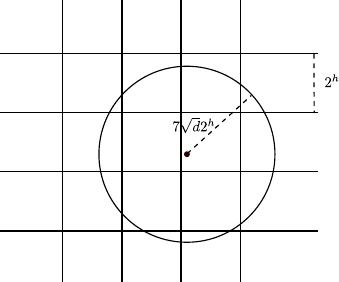}
\caption{The ball with radius $7\sqrt d 2^{h}$ intersects at most $(2+7\sqrt d)^d$ tiles of size $2^{h}$.}
\label{fig:2}
\end{figure}

Put this into \eqref{eq:xyzw}, note that $2^{-\beta h}=2^{h\log_2(2^dpq)}=(2^dpq)^h$,
we have
\begin{equation}\label{eq:R2}
\begin{aligned}
&\sum_{h=0}^m\sum_{(x,y,z,w)\in\Theta_h}
\sum_{i,i',j,j'}2^{-\beta h}a_{ij}a_{i'j'}\1_{(i,i',j,j')\succeq(x,y,z,w)}
\\
\lesssim&\sum_{h=0}^m2^{-\beta h+d(m-h)}\sum_{|x|=m-h,|z|=n-h}
a^2_{[x][z]}
=2^{dm}\sum_{h=0}^m(pq)^{h}\sum_{|x|=m-h,|z|=n-h}
a^2_{[x][z]}.
\end{aligned}
\end{equation}
Finally, \eqref{eq:good_side} follows by combining \eqref{eq:R0}, \eqref{eq:R1}, and \eqref{eq:R2}. 
\end{proof}

Now we are ready to give upper bound in Theorem \ref{thm:main}. For later usage, we extend the domain where the target set stays from $[-2^{n-1},2^{n-1})^d\cap\mathbb Z^d$ to $[-2^{n-1}-2^{m-1},2^{n-1}+2^{m-1})^d\cap\mathbb Z^d$.
\begin{lemma}\label{lem:main_upper}
Let $d\ge 1$, let $p,q\in(0,1)$ with property \eqref{eq:nonempty}.
For every $1\le m\le n$ and every finite set 
$A\subset[-2^{n-1}-2^{m-1},2^{n-1}+2^{m-1})^d\cap\mathbb Z^d$,
\begin{equation*}
\mathbb P((Q_d(p;m)+\widehat Q_d(q;n))\cap A\ne\emptyset)
\lesssim (2^dp)^{m\vee\log_2\diam(A)}q^n\Cap_\beta(A).
\end{equation*}
Here $\lesssim$ only depends on $d,p,q$ (independent of $m,n,A$).
\end{lemma}
\begin{proof}
By Lemma \ref{lem:up1} and Lemma \ref{lem:up2},
\[
{\mathbb P\pars*{
Q_d(p;m)+\widehat Q_d(q;n)\cap A\ne\emptyset
}}
\lesssim
(2^dp)^{m\vee\log_2\diam(A)}q^n\pars*{\sum_{i,i',j,j'}a_{ij}a_{i'j'}|i\oplus j-i'\oplus j'|^{-\beta}}^{-1}.
\]
Let
\[
\mu(x)=\sum_{i,j}a_{ij}\1_{\mathscr R_m(i)+\mathscr R_n(j)=x},
\]
then $\mu$ is a probability measure supported on $A$, and the equation above implies
\begin{align*}
&\mathbb P((Q_d(p;m)+\widehat Q_d(q;n))\cap A\ne\emptyset)\\
\lesssim&
(2^dp)^{m\vee\log_2\diam(A)}q^n\pars*{\sum\mu(x)\mu(y)|x-y|^{-\beta}}^{-1}\\
\le&
(2^dp)^{m\vee\log_2\diam(A)}q^n\Cap_\beta(A).
\end{align*}

\end{proof}

\begin{proof}[Proof of Theorem \ref{thm:main}]
It is now immediate by Lemma \ref{lem:main_lower} and Lemma \ref{lem:main_upper}.
\end{proof}

\section{Applications}

\subsection{Continuous case, proof of Corollary \ref{cor:main_continuous}}
\label{sec:3.1}
\begin{proof}[Proof of Corollary \ref{cor:main_continuous}]
The first conclusion \eqref{eq:main_discrete+} follows immediately from \eqref{eq:FP_cap} and Theorem \ref{thm:main} by taking $m=n$.

Below we consider the continuous case.
For any closed set $\Lambda\subset[-\frac12,\frac12]^d$, denote 
\[
T_k(\Lambda):=
\bigcup_{\setof{x}{2^kx\in \mathbb Z^d,x+[-2^{1-k},2^{1-k}]^d\cap\Lambda\ne\emptyset}}(x+[-2^{1-k},2^{1-k}]^d)
\]
then 
$T_k(\Lambda)\subset[-\frac12-2^{1-k},\frac12+2^{1-k}]^d$, and it is not hard to check that 
$(T_k(\Lambda))_k$ is a decreasing sequence of sets with
\[
\bigcap_{k\ge 0}T_k(\Lambda)=\Lambda.
\]

We claim that
\[
(Q_d(p)+\widehat Q_d(q))\cap \Lambda\ne\emptyset
\iff
\forall k,(Q_d(p;k)+\widehat Q_d(q;k))\cap 2^kT_k(\Lambda)\ne\emptyset.
\]

On one hand, if $(Q_d(p)+\widehat Q_d(q))\cap \Lambda\ne\emptyset$,
we take $a\in Q_d(p),b\in \widehat Q_d(q)$ with $a+b\in\Lambda$. Then by Definition \ref{def:1}, for each $k$ there is $a_k\in Q_d(p;k),b_k\in\widehat Q_d(q;k)$ such that
\[
2^ka\in a_k+[-1,1]^d,\;2^kb\in b_k+[-1,1]^d.
\]
This implies
\[
2^k(a+b)\in (a_k+b_k)+[-2,2]^d,
\]
so $2^{-k}(a_k+b_k)\in T_k(\Lambda)$.

On the other hand, if we have $a_k\in Q_d(p;k)$, $b_k\in \widehat Q_d(q;k)$ with $a_k+b_k\in2^k T_k(\Lambda)$ for every $k$, by looking for a subsequence we can assume $2^{-k}a_k\rightarrow a$, $2^{-k}b_k\rightarrow b$. 
Notice that $(2^{-k}\overline {Q_d(p;k)})$ is a decreasing sequence of closed set. For each fixed $k$, $2^{-k}\overline {Q_d(p;k)}$ contains $a_k,a_{k+1},\dots$, therefore also contains their limit point $a$. Hence $Q_d(p)$ contains $a$. Similarly we have $b\in \widehat Q_d(q)$ and $a+b\in \Lambda$.

This allows us to conclude that
\begin{align*}
\mathbb P\pars*{(Q_d(p)+\widehat Q_d(q))\cap \Lambda\ne\emptyset}
=\lim_{k\rightarrow\infty}\mathbb P\pars*{(Q_d(p;k)+\widehat Q_d(q;k))\cap 2^kT_k(\Lambda)\ne\emptyset}.
\end{align*}
Then we can apply Theorem \ref{thm:main} to $m=n=k$, $A=2^kT_k(\Lambda)$ (use Lemma \ref{lem:main_lower} and Lemma \ref{lem:main_upper} directly if $2^{-k}A\not\subset[-\frac 12,\frac 12]^d$ slightly exceeds the boundary)  and obtain
\begin{align*}
\mathbb P\pars*{(Q_d(p)+\widehat Q_d(q))\cap \Lambda\ne\emptyset}
&\asymp\lim_{k\rightarrow\infty}2^{-\beta k}\Cap_\beta(2^kT_k(\Lambda))\\
&=\lim_{k\rightarrow\infty}\Cap_\beta(T_k(\Lambda))
=\Cap_\beta(\Lambda),
\end{align*}
where $\asymp$ is independent of $k$.
This proves the first part of \eqref{eq:main_continuous}.
The second part of \eqref{eq:main_continuous} is immediate by \cite[Corollary 4.3]{peres1996intersection}. 

\end{proof}

\subsection{Changing parameter in capacity, proof of Corollary \ref{cor:cap_a_b}}\label{sec:3.2}

\begin{proof}[Proof of Corollary \ref{cor:cap_a_b}]
By \eqref{eq:FP_cap} (and symmetry in $Q$ and $-Q$),
\begin{align*}
&\mathbb P((Q_d(p;m)+\widehat Q_d(q;n))\cap A\ne\emptyset)\\
=&\mathbb P(\widehat Q_d(q;n)\cap (A+Q_d(p;m))\ne\emptyset)\\
\asymp&q^n\mathbb E[\Cap_{-\log_2 q}(A+Q_d(p;m))],
\end{align*}
as long as $n$ is large enough such that $A+Q_d(p;m)\subset[-2^{n-1},2^{n-1})^d$.

Combining this with Theorem \ref{thm:main}, we have
\begin{equation*}
(2^{d}p)^m\lesssim
\frac{\mathbb E[\Cap_{-\log_2 q}(A+Q_d(p;m))]}{\Cap_\beta(A)}
\lesssim (2^dp)^{m\vee\log_2\diam(A)}.
\end{equation*}
The first conclusion follows by taking $p=2^{b-a-d},q=2^{-b}$.

The continuous version follows from a similar argument as proof of Corollary \ref{cor:main_continuous} with $A=2^mT_m(\Lambda)$.
\end{proof}

\subsection{Consequences for sum of general random sets, proof of Theorem \ref{thm:rw+rw}}
\label{sec:3.3}

\begin{lemma}\label{lem:divide}
Let $k\ge 1$, $\beta>0$, let $A=A_1\cup\dots\cup A_k$. Then
\begin{align*}
\Cap_\beta(A)\asymp\Cap_\beta(A_1)+\dots+\Cap_\beta(A_k)
\asymp\max\braces{\Cap_\beta(A_1),\dots,\Cap_\beta(A_k)},
\end{align*}
where the constant in $\asymp$ only depends on $k$ (independent of $\beta,A,A_1,\dots,A_k$).
\end{lemma}
\begin{proof}
Capacity is clearly monotone, so $\Cap_\beta(A)\ge\max\braces{\Cap_\beta(A_1),\dots,\Cap_\beta(A_k)}$. 
Capacity is also subadditive (see for instance \cite[Proposition 6.5.3]{lawler2010random}), so
\[
\Cap_\beta(A)\le\Cap_\beta(A_1)+\dots+\Cap_\beta(A_k)
\le k\max\braces{\Cap_\beta(A_1),\dots,\Cap_\beta(A_k)}.
\]
\end{proof}

We are now ready to prove Theorem \ref{thm:rw+rw}.
\begin{theorem}
Let $d,k\ge 1$, let $0<\alpha_i,\beta_i<d$, $i=1,2,\dots,k$ such that 
\begin{align}
\label{eq:condition_gamma}
\gamma:=\alpha_1+\dots+\alpha_k-(k-1)d>0,
\end{align}
\[
\beta_1+\dots+\beta_k-(k-1)d>0.
\]
Let $\mathcal R_1,\dots,\mathcal R_k$ be independent random sets in $\mathbb Z^d$ with the following property: for every finite set $B\subset\mathbb Z^d$ containing the origin and every $|x|\ge 2\diam(B)$,
we have
\begin{equation}
\label{eq:condition_B}
{\mathbb P(\mathcal R_i\cap (x+B)\ne\emptyset)}
\asymp 
|x|^{-\beta_i}
\Cap_{\alpha_i}(B),\quad i=1,2,\dots,k.
\end{equation}
Let $A\subset\mathbb Z^d$ be any finite set containing the origin. When
\begin{align}
\label{eq:rA}
|x|\ge 2\diam(A)+\max_{1\le i\le k}\1_{\alpha_i\le\beta_i}\diam(A)^{\frac {d-\alpha_i}{d-\beta_i}}\log^{\frac{1}{d-\beta_i}}(\diam(A)+1),
\end{align}
then
\begin{align*}
\mathbb P((\mathcal R_1+\dots+\mathcal R_k)\cap (x+A)\ne\emptyset)
\asymp
|x|^{(k-1)d-\beta_1-\dots-\beta_k}\Cap_{\gamma}(A),
\end{align*}
where the constant in $\asymp$ only depends on $d,k$ and $(\alpha_i),(\beta_i)$ (independent of $A,x$).

\end{theorem}

\begin{proof}
For each $n$, abbreviate 
\[
\Psi_n=[-2^{n-1},2^{n-1}]^d\setminus [-2^{n-2},2^{n-2}]^d,
\]
and we assume all different instances of fractal percolations appearing in this proof to be independent.

Take $n=\lceil\log_2 |x|\rceil$. 
By \eqref{eq:rA}, we know that $x+A\subset\Psi_{n+1}\cup\Psi_n\cup\Psi_{n-1}$. By Lemma \ref{lem:divide}, it suffices to assume $x+A\subset\Psi_n$ (otherwise cut the set into pieces in $\Psi_{n+1},\Psi_n,\Psi_{n-1}$, and apply the argument for $n+1,n,n-1$), and it remains to show that
\begin{align*}
\mathbb P((\mathcal R_1+\dots+\mathcal R_k)\cap (x+A)\ne\emptyset)
\asymp
2^{((k-1)d-\beta_1-\dots-\beta_k)n}\Cap_{\gamma}(A).
\end{align*}

Moreover, by \eqref{eq:condition_B} and Lemma \ref{lem:divide}, it is not hard to observe that for every $B\subset\Psi_n$, 
\begin{align}\label{eq:divide}
\mathbb P(\mathcal R_i\cap B\ne\emptyset)\asymp 2^{-n\beta_i}\Cap_{\alpha_i}(B).
\end{align}

We first prove the lower bound. Let 
\[
M\oplus N:=(M+N)\cap\Psi_n,
\]
then
\begin{align*}
\mathbb P((\mathcal R_1+\dots+\mathcal R_k)\cap (x+A)\ne\emptyset)
&\ge
\mathbb P((\mathcal R_1\oplus\dots\oplus\mathcal R_k)\cap (x+A)\ne\emptyset)\\
&=\mathbb P(\mathcal R_1\cap(-\mathcal R_2\oplus\dots\oplus-\mathcal R_k\oplus (x+A))\ne\emptyset).
\end{align*}
Notice that 
\[
-\mathcal R_2\oplus\dots\oplus-\mathcal R_k\oplus (x+A)\subset\Psi_n,
\]
by \eqref{eq:divide} and \eqref{eq:FP_cap},
\begin{equation}\label{eq:iterate_1}
\begin{aligned}
&\mathbb P(\mathcal R_1\cap(-\mathcal R_2\oplus\dots\oplus-\mathcal R_k\oplus (x+A))\ne\emptyset)\\
\asymp &2^{-\beta_1 n}\Cap_{\alpha_1}(-\mathcal R_2\oplus\dots\oplus-\mathcal R_k\oplus (x+A))\\
\asymp &2^{(\alpha_1-\beta_1) n}\mathbb P(Q_d(2^{-\alpha_1};n)\cap (-\mathcal R_2\oplus\dots\oplus-\mathcal R_k\oplus (x+A))\ne\emptyset)\\
= &2^{(\alpha_1-\beta_1) n}\mathbb P((Q_d(2^{-\alpha_1};n)\oplus\mathcal R_2\oplus\dots\oplus\mathcal R_k)\cap (x+A))\ne\emptyset). 
\end{aligned}
\end{equation}

Iterate the same procedure \eqref{eq:iterate_1} for $\mathcal R_2,\dots\mathcal R_k$, we have
\begin{align*}
&\mathbb P((\mathcal R_1\oplus\dots\oplus\mathcal R_k)\cap (x+A)\ne\emptyset)
\gtrsim
2^{n\sum_{i=1}^k(\alpha_i-\beta_i)}\mathbb P((Q_1\oplus\dots\oplus Q_k)\cap (x+A)\ne\emptyset),
\end{align*}
where we denote $Q_i=Q_d(2^{-\alpha_i};n)$ for simplicity.

Notice that for a set $B\subset[-2^{n-1},2^{n-1}]^d$, 
\[
\mathbb P((Q_1\oplus Q_2)\cap B\ne\emptyset)
=\mathbb P((Q_1+ Q_2)\cap B\ne\emptyset),
\]
by Corollary \ref{cor:cap_a_b} (the condition $\alpha_i\in(0,d)$ and \eqref{eq:condition_gamma} guarantees that $2^{d-\alpha_1-\alpha_2}<1$), we have
\begin{equation}\label{eq:Q+Q_lower}
\begin{aligned}
&\mathbb P((Q_1\oplus\dots \oplus Q_k)\cap (x+A)\ne\emptyset)\\
=&\mathbb P((Q_1\oplus Q_2)\cap ((x+A)\oplus Q_3\oplus\dots\oplus Q_k)\ne\emptyset)\\
=&\mathbb P((Q_1+ Q_2)\cap ((x+A)\oplus Q_3\oplus\dots\oplus Q_k)\ne\emptyset)\\
\asymp&\mathbb P((Q_d(2^{d-\alpha_1-\alpha_2};n))\cap ((x+A)\oplus Q_3\oplus\dots\oplus Q_k)\ne\emptyset)\\
=&\mathbb P((Q_d(2^{d-\alpha_1-\alpha_2};n)\oplus Q_3\oplus\dots\oplus Q_k)\cap (x+A)\ne\emptyset).
\end{aligned}
\end{equation}
This reduces the number of fractal percolations from $k$ to $k-1$.
The lower bound of Theorem \ref{thm:rw+rw} now follows by iteration of \eqref{eq:Q+Q_lower} and the fact $\Cap(A)=\Cap(x+A)$.

For the upper bound, by union bounds,
\begin{equation}\label{eq:union_bounds}
\begin{aligned}
&\mathbb P((\mathcal R_1+\dots+\mathcal R_k)\cap (x+A)\ne\emptyset)\\
=&
\mathbb P(\mathcal R_1\cap(-\mathcal R_2+\dots+-\mathcal R_k+ (x+A))\ne\emptyset)\\
\le&\sum_{j}
\mathbb P(\Psi_j\cap\mathcal R_1\cap(-\mathcal R_2+\dots+-\mathcal R_k + (x+A))\ne\emptyset).
\end{aligned}
\end{equation}
We need this step because even though $x+A\subset\Psi_n$, $-\mathcal R_2+\dots+-\mathcal R_k + (x+A)$ can be unbounded.
By \eqref{eq:divide},
\begin{align*}
&\mathbb P(\Psi_j\cap\mathcal R_1\cap(-\mathcal R_2+\dots+-\mathcal R_k + (x+A))\ne\emptyset)\\
\asymp&2^{-\beta_1 j}\Cap_{\alpha_1}((-\mathcal R_2+\dots+-\mathcal R_k + (x+A))\cap\Psi_j).
\end{align*}
By \eqref{eq:FP_cap}, we conclude that
\begin{align*}
&\mathbb P((\mathcal R_1+\dots+\mathcal R_k)\cap (x+A)\ne\emptyset)\\
\lesssim&\sum_j 2^{-\beta_1 j}\Cap_{\alpha_1}((-\mathcal R_2+\dots+-\mathcal R_k + (x+A))\cap\Psi_j)\\
\lesssim&\sum_j 2^{(\alpha_1-\beta_1)j}\mathbb P\pars{
Q_d(2^{-\alpha_1};j)\cap
(-\mathcal R_2+\dots+-\mathcal R_k + (x+A))\ne\emptyset
}\\
=&\sum_j 2^{(\alpha_1-\beta_1)j}\mathbb P\pars{
(Q_d(2^{-\alpha_1};j)+\mathcal R_2+\dots+\mathcal R_k)\cap
(x+ A)\ne\emptyset
}.
\end{align*}
Iterate for each $\mathcal R_i$, we have
\begin{equation}\label{eq:iterate_2}
\begin{aligned}
&\mathbb P((\mathcal R_1+\dots+\mathcal R_k)\cap A\ne\emptyset)\\
\lesssim&\sum_{j_1,\dots,j_k}2^{\sum_i(\alpha_i-\beta_i)j_i}\mathbb P\pars{
(Q_d(2^{-\alpha_1};j_1)+\dots+Q_d(2^{-\alpha_k};j_k))\cap
(x+A)\ne\emptyset}.
\end{aligned}
\end{equation}
Since $\mathcal R_1,\dots,\mathcal R_k$ are exchangeable, it suffices to study the case $j_1\ge\dots\ge j_k$, i.e.
\[
\sum_{j_1,\dots,j_k}\1_{j_1\ge\dots\ge j_k}2^{\sum_i(\alpha_i-\beta_i)j_i}\mathbb P\pars{
(Q_d(2^{-\alpha_1};j_1)+\dots+Q_d(2^{-\alpha_k};j_k))\cap
(x+A)\ne\emptyset}
\]
Note that in order that the intersection happens, sum of the blocks $[-2^{j_i-1},2^{j_i-1}]^d$ must reach $x+A\subset\Psi_n$, hence  
\begin{align}\label{eq:i1}
k2^{j_1}\ge 2^{j_1}+\dots+2^{j_k}\ge 2^{n-2}.
\end{align}
By \eqref{eq:FP_cap} (and the fact that capacity is translational invariant),
\begin{equation*}
\begin{aligned}
&\mathbb P\pars{
(Q_d(2^{-\alpha_1};j_1)+\dots+Q_d(2^{-\alpha_k};j_k))\cap
(x+A)\ne\emptyset}\\
\lesssim&2^{-\alpha_1 j_1}
\mathbb E\bracks{
\Cap_{\alpha_1}(Q_d(2^{-\alpha_2};j_2)+\dots+Q_d(2^{-\alpha_k};j_k)+A)}.
\end{aligned}
\end{equation*}
Now $j_2$ becomes the largest among the remaining index, and we can use Corollary \ref{cor:cap_a_b} and \eqref{eq:divide} to obtain
\begin{equation*}
\begin{aligned}
&\mathbb E\bracks{
\Cap_{\alpha_1}(Q_d(2^{-\alpha_2};j_2)+\dots+Q_d(2^{-\alpha_k};j_k)+A)}\\
\lesssim&2^{(d-\alpha_2)(j_2\vee\log_2\diam(A))}
\mathbb E\bracks{
\Cap_{\alpha_1+\alpha_2-d}(Q_d(2^{-\alpha_3};j_3)+\dots+Q_d(2^{-\alpha_k};j_k)+A)}.
\end{aligned}
\end{equation*}
Iterate this argument for $j_3,\dots,j_k$, given that as $j_1\ge j_2\ge\dots\ge j_k$, we have (recall $\gamma=\alpha_1+\dots+\alpha_k-(k-1)d>0$)
\begin{equation}\label{eq:iterate_3}
\begin{aligned}
&\mathbb P\pars{
(Q_d(2^{-\alpha_1};j_1)+\dots+Q_d(2^{-\alpha_k};j_k))\cap
A\ne\emptyset}\\
\lesssim&
2^{-\alpha_1j_1}2^{\sum_{i\ge 2}(d-\alpha_i)(j_i\vee\log_2\diam(A))}
\Cap_{\gamma}(A).
\end{aligned}
\end{equation}
Therefore,
\begin{align*}
&\sum_{j_1,\dots,j_k}\1_{j_1\ge\dots\ge j_k}
2^{\sum_i(\alpha_i-\beta_i)j_i}
2^{-\alpha_1j_1}2^{\sum_i(d-\alpha_i)(j_i\vee\log_2\diam(A))}
\Cap_{\gamma}(A)\\
\le&\sum_{j_1\ge n-2-\log_2k} 2^{-\beta_1j_1}\Cap_\gamma(A)\cdot\prod_{i=2}^k\pars*{\sum_{j\le j_1}2^{(\alpha_i-\beta_i)j+(d-\alpha_i)(j\vee\log_2\diam(A))}}\\
\lesssim &\sum_{j_1\ge n-2-\log_2k} 2^{-\beta_1j_1}\Cap_\gamma(A)\cdot\prod_{i=2}^k\pars*{2^{(d-\beta_i)j_1}+\log_2\diam(A)2^{(d-\alpha_i)\log_2\diam(A)}\1_{\alpha_i\le \beta_i}}.
\end{align*}
By \eqref{eq:rA} and \eqref{eq:i1}, the term on $\diam(A)$ is negligible, hence 
\begin{align*}
&\sum_{j_1,\dots,j_k}\1_{j_1\ge\dots\ge j_k}
2^{\sum_{i=1}^k(\alpha_i-\beta_i)j_i}
2^{-\alpha_1j_1}2^{\sum_{i=2}^k(d-\alpha_i)(j_i\vee\log_2\diam(A))}
\Cap_{\gamma}(A)\\
\lesssim &\sum_{j_1\ge n-2-\log_2k} 2^{-\beta_1j_1}\Cap_\gamma(A)\cdot{2^{\sum_{i=2}^k(d-\beta_i)j_1}}
\asymp 2^{n((k-1)d-\beta_1-\dots-\beta_k)}\Cap_\gamma(A).
\end{align*}
The conclusion follows by combining this with \eqref{eq:iterate_2} and \eqref{eq:iterate_3}.
\end{proof}

\begin{remark}
The proof above relies on the intuition that hitting events are dominated by that of a typical scale $\Psi_n$. When $\diam(A)$ is too large, this intuition may fails. Take for instance $A$ to be a ball of radius $c|x|$, and take $(\mathcal R_i)$ as ranges of simple random walks, then one single range hits $A$ with positive probability, so \eqref{eq:union_bounds} is no longer reasonable.
\end{remark}

\subsection{Remark on the case \texorpdfstring{$\beta\le 0$}{} in \texorpdfstring{\eqref{eq:nonempty}}{}}
\label{sec:log_cap}

If $2^dpq\ge1$, by \eqref{eq:second_moment_L} we have
\[
{\mathbb P((Q_d(p;m)+\widehat Q_d(q;n))\cap A\ne\emptyset)}
\gtrsim 
\begin{cases}
q^{n-m},\,2^dpq>1,\\
m^{-1} q^{n-m},\,2^dpq=1,
\end{cases}
\]
as long as $A$ is nonempty.
Moreover, 
\begin{align*}
{\mathbb P((Q_d(p;m)+\widehat Q_d(q;n))\cap A\ne\emptyset)}
=&\mathbb P({\widehat Q_d(q;n)\cap (A+Q_d(p;m))\ne\emptyset)}\\
\lesssim&q^{-n}\Cap_{\log_2 q}(Q_d(p;m)-A)\\
\lesssim&q^{-n}(2^m+\diam(A))^{\log_2 q}\\
\lesssim&q^{n-(m\vee \log_2\diam(A))}.
\end{align*}

When $m\gtrsim \log_2\diam(A)$
and $2^dpq>1$, we obtain 
\begin{align*}
\mathbb P((Q_d(p;m)+\widehat Q_d(q;n))\cap A\ne\emptyset)\asymp q^{n-m}.
\end{align*}
When $m\ll\log_2\diam(A)$ or $m\gtrsim \log_2\diam(A)$
and $2^dpq=1$, the upper and lower bounds do not meet, yet in fact they are both optimal (by taking $A=\{0\},\{0,2^m\},$ or $[-2^m,2^m)^d\cap\mathbb Z^d$), which indicates that only knowing information on diameter of $A$ is insufficient for a sharp estimate. In particular, we conjecture that the case $2^dpq=1$ can be related to log-capacity of the target set $A$.

% \bibliographystyle{plain}
% \bibliography{1}
\end{document}